\documentclass[12pt,a4paper]{amsart}

\usepackage{amsmath, nicefrac, amsthm, verbatim, amsfonts, amssymb, xcolor, bbm}
\usepackage{graphics, xspace, enumerate}
\usepackage[a4paper,margin=3cm]{geometry}

\usepackage[bb=boondox]{mathalfa}

\usepackage{graphicx}
\usepackage[colorlinks=true,citecolor=green,urlcolor=green,linkcolor=green,bookmarksopen=true,unicode=true,pdffitwindow=true]{hyperref}
\usepackage[english]{babel}
\usepackage[languagenames,fixlanguage]{babelbib}

\hypersetup{pdfauthor={}}
\hypersetup{pdftitle={CLT for the capacity of the range of stable random walks}}

\theoremstyle{plain}
\newtheorem{theorem}{Theorem}[section]
\newtheorem{corollary}[theorem]{Corollary}
\newtheorem{lemma}[theorem]{Lemma}
\newtheorem{proposition}[theorem]{Proposition}
\theoremstyle{definition}

\newcommand {\R} {\ensuremath{\mathbb{R}}}
\newcommand {\ZZ} {\ensuremath{\mathbb{Z}}}
\newcommand {\N} {\ensuremath{\mathbb{N}}}

\newcommand{\Capa}{\operatorname{Cap}}
\newcommand{\Var}{\operatorname{Var}}

\newcommand{\bbZ}{\mathbb{Z}}

\newcommand{\bbP}{\mathbb{P}}
\newcommand{\bbE}{\mathbb{E}}
\newcommand{\bbjedan}{\mathbbm{1}}

\newcommand{\calC}{\mathcal{C}}

\newcommand{\calJ}{\mathcal{J}}
\newcommand{\calE}{\mathcal{E}}
\newcommand{\calF}{\mathcal{F}}

\newcommand{\calP}{\mathcal{P}}
\newcommand{\calR}{\mathcal{R}}
\newcommand{\calN}{\mathcal{N}}

\newcommand{\aps}[1]{\vert #1 \vert}

\newcommand{\floor}[1]{\lfloor #1 \rfloor}

\newcommand{\obl}[1]{\bigg( #1 \bigg)}

\newcommand{\ugl}[1]{\bigg[ #1 \bigg]}

\newcommand{\vit}[1]{\bigg\{ #1 \bigg\}}

\numberwithin{equation}{section}

\title[CLT for the capacity of the range of stable random walks]{CLT for the capacity of the range\\ of stable random walks}

\author[W.\ Cygan]{Wojciech Cygan}
\address[Wojciech Cygan]{Institut f\"{u}r Mathematische Stochastik\\Technische Universit\"{a}t Dresden\\Dresden\\Germany
\& 
Instytut Matematyczny\\Uniwersytet Wroc\l{}awski\\ Wroc\l{}aw\\ Poland}
\email{wojciech.cygan@uwr.edu.pl}

\author[N.\ Sandri\'{c}]{Nikola Sandri\'{c}}
\address[Nikola\ Sandri\'{c}]{Department of Mathematics\\University of Zagreb\\ Zagreb\\Croatia}
\email{nsandric@math.hr}

\author[S.\ \v{S}ebek]{Stjepan\ \v{S}ebek}
\address[Stjepan\ \v{S}ebek]{Department of Applied Mathematics\\
	Faculty of Electrical Engineering and Computing\\
	University of Zagreb\\ 
 Zagreb\\ 
	Croatia}
\email{stjepan.sebek@fer.hr}


\begin{document}
\allowdisplaybreaks[4]

\begin{abstract}
    In this article, we establish a central limit theorem for the capacity of the range process for a class of $d$-dimensional symmetric $\alpha$-stable random walks with the index satisfying $d /\alpha> 5 /2$. Our approach is based on controlling the limit behavior of the variance of the capacity of the range process which then allows us to apply the Lindeberg-Feller theorem.
       
\end{abstract}

\maketitle

%
%
%
%

\section{Introduction}

Let $(\Omega,\calF,\bbP)$ be a probability space, and let $\{X_i\}_{i\in\N}$ be a sequence of i.i.d. $\bbZ^d$-valued random variables defined on $(\Omega,\calF,\bbP)$, where $d\ge1$  and $\bbZ^d$ stands for the $d$-dimensional integer lattice. For $x\in\ZZ^d$ define $S_0=x$ and $S_n=S_{n-1}+X_n$, $n\ge1$. The stochastic process $\{S_n\}_{n\ge0}$ is called a  $\bbZ^d$-valued random walk starting from $x$.

Throughout the article we will often rely on the Markovian nature of $\{S_n\}_{n\ge0}$, therefore we need to allow arbitrary initial conditions of the underlying probability measure. For this purpose we redefine the probability space in the following way.  Put $\bar{\Omega}=\bbZ^d\times\Omega,$ $\bar{\calF}=\calP(\bbZ^d)\otimes\calF$, and $\bbP_x=\delta_x\times\bbP$ for $x\in\ZZ^d.$ 
A random variable $X$ on $(\Omega,\calF,\bbP)$ is extended automatically to   
$(\bar\Omega,\bar\calF,\{\bbP_x\}_{x\in\ZZ^d})$ by the rule $X(x,\omega)=X(\omega)$ for $x\in\ZZ^d$ and $\omega\in\Omega$. Further, define $S_0:\bar{\Omega}\to\ZZ^d$ by $S_0(x,\omega)=x$ for $x\in\ZZ^d$ and $\omega\in\Omega$. Clearly, $\bbP_x(S_0=x)=1$, and  for each  $x\in\ZZ^d$ the process $\{S_n\}_{n\ge0}$ is  a  $\bbZ^d$-valued random walk  on  $(\bar\Omega,\bar\calF,\bbP_x)$ starting from $x$. Also, it is a (strong) Markov process (with respect to the corresponding natural filtration). Observe that the corresponding transition probabilities are given by 
$$p_n(x,y)=\bbP_x(S_n=y)=\bbP_0(S_n=y-x),\qquad n\ge0,\ x,y\in\ZZ^d.$$ 
From the above relation  we immediately see that there are functions $\{p_n\}_{n\ge0}$ such that $p_n(y-x)=p_n(x,y)=p_n(0,y-x)$, $n\ge0$, $x,y\in\ZZ^d$.
Also, for notational simplicity we write $(\Omega,\calF,\{\bbP_x\}_{x\in\ZZ^d})$ instead of $(\bar\Omega,\bar\calF,\{\bbP_x\}_{x\in\ZZ^d})$, and when $x=0$ we suppress the index $0$ and write $\bbP$ instead of $\bbP_0$. We denote by 
\begin{align*}
G(x,y)=\sum_{n\ge0}p_n(y-x),\qquad x,y\in\ZZ^d
\end{align*}
the Green function of $\{S_n\}_{n\ge0}$. Due to the spatial homogeneity of $\{S_n\}_{n\ge0}$, we sometimes write $G(y-x)$ instead of $G(x,y)$. 
Recall  that $\{S_n\}_{n\ge0}$ is called transient if $G(0)<\infty$; otherwise it is called recurrent.

The main aim of this article is to establish a central limit theorem for the capacity of the range process of  $\{S_n\}_{n\ge0}$.
Recall that the range process $\{\calR_n\}_{n\ge0}$ is defined as the random set
\begin{align*}
\calR_n = \{S_0,\dots,S_n\},\qquad n\ge 0.
\end{align*}
For $1\le m\le n$ we use notation $\calR[m,n]=\{S_m,\dots,S_n\}.$
The capacity of a set $A\subseteq\ZZ^d$ (with respect to any transient random walk $\{S_n\}_{n\ge0}$) is defined as 
\begin{align*}
{\rm Cap}(A)=\sum_{x\in A}\bbP_x(T^+_A=\infty).
\end{align*}
Here, $T^+_A$ denotes the first return time of $\{S_n\}_{n\ge0}$ to the set $A$, that is, 
\begin{align*}
T^+_A=\inf\{n\ge1:S_n\in A\}.
\end{align*}
 Also, when $A=\{x\}$, $x\in\ZZ^d$, we write $T^+_x$ instead of $T^+_{\{x\}}$. We are interested in the long-time behaviour of the process  $\{\calC_n\}_{n\ge0}$ defined as 
\begin{align*}
\calC_n={\rm Cap}(\calR_n ).
\end{align*}
Before stating the main result, we  introduce and discuss the  assumptions which we impose on the random walk  $\{S_n\}_{n\ge0}$.

\begin{itemize}
	\item [(\textbf{A1})] $\{S_n\}_{n\ge0}$ is   aperiodic, that is, the smallest additive subgroup generated by the set ${\rm supp}\, p_1=\{x\in\mathbb{Z}^d:\, p_1(x)>0\}$ is equal to $\ZZ^d$.
	
	\medskip
		\item [(\textbf{A2})] $\{S_n\}_{n\ge0}$ is  symmetric and  strongly transient.
	
	\medskip
	
	\item [(\textbf{A3})] The step $X_1$ of the random walk $\{S_n\}_{n\ge0}$ belongs to the domain of  attraction of a  non-degenerate $\alpha$-stable random law with $0<\alpha\le2$, meaning that there exists a regularly varying function $b(x)$ with index $1/\alpha$ such that 
	$$\frac{S_n}{b(n)}\xrightarrow[n\nearrow\infty]{\text{(d)}}U_\alpha,$$ 
where $U_\alpha$ is an  $\alpha$-stable random variable on $\R^d$ and $\ \xrightarrow[]{\text{(d)}}$ stands for the convergence in distribution.

\medskip

\item [(\textbf{A4})] $\{S_n\}_{n\ge0}$ admits one-step loops, that is, $p=p_1(0)>0$.

\end{itemize}

Let us  remark that assumption (\textbf{A1}) is not restrictive in any sense. Namely, if $\{S_n\}_{n\ge0}$ is not aperiodic, we can then perform our analysis (and obtain the same results) on the smallest additive subgroup of $\ZZ^d$ generated by ${\rm supp}\,p_1$ (see \cite[pp 20]{Spitzer}). 

To discuss (\textbf{A2}) and (\textbf{A3}) we recall that a transient random walk $\{S_n\}_{n\ge0}$ is called strongly transient if $\sum_{n\ge1}n\,p_n(0)<\infty$; otherwise it is called weakly transient.   
It is known that every transient random walk is either strongly or weakly transient (see \cite{Sato}). Under (\textbf{A3}),   $\{S_n\}_{n\ge0}$ is transient if $d/\alpha>1$ and strongly transient if $d/\alpha>2$ (see \cite[Theorem 3.4]{Sato}, cf.\ also  \cite[Theorem 7]{Takeuchi}).  
The notion of strong transience was first introduced  in \cite{Port} for Markov chains and was later used in \cite{Jain-Orey} in the context of the limit behavior of the range of random walks.
Actually, in \cite{Jain-Orey} a slightly different definition of 
 strong (weak) transience has been used: a transient random walk $\{S_n\}_{n\ge0}$ is called strongly  transient if $\sum_{n\ge1}n\,\bbP(T^+_{0}=n)<\infty$; otherwise it is called weakly transient. For reader's convenience we show that these two definitions are equivalent. Indeed, starting from the following classical identity (see \cite{Spitzer})
\begin{align*}
p_k(0)= \sum_{j=1}^k \bbP(T^+_{0}=j)p_{k-j}(0),\qquad k\ge1,
\end{align*}
we easily obtain that
\begin{align*}
\sum_{n=1}^\infty n\,p_n(0) \big( 1-\sum_{j=1}^\infty \bbP(T^+_{0}=j)\big) = G(0)\sum_{n=1}^\infty n\,\bbP(T^+_{0}=n),
\end{align*}
and whence both series must converge simultaneously. It is a well-known fact that the condition $G(0)<\infty$ forces $\bbP(T^+_0<\infty)=\sum_{n=1}^\infty \bbP(T^+_{0}=n) <1$.

We remark that the strong transience assumption is very natural in this context. Namely, it ensures that the range process $\{\calR_n\}_{n\ge0}$ grows fast enough which allows us to conclude that the    
 limiting distribution in Theorem \ref{CLT} below is not degenerated, in other words, the constant $\sigma_d$ does not vanish.

Validity of condition (\textbf{A3}) can be checked with the aid of the regular variation of the tails of the step $X_1$. In the one-dimensional case we refer to \cite[Theorem 8.3.1]{BGT_book}, and in the multidimensional case analogous results can be found in \cite[Section 5.4.2]{Resnick}, \cite[Theorem 4.2]{Rvaceva} and the references therein. Conditions for stability of a random vector in terms of its one-dimensional projections can be found in \cite{Gupta}.

Finally,  assumption  (\textbf{A4}) is of technical nature only. By using a random time-change argument and loop decomposition technique, it allows us to conclude that the limit in \eqref{CLT} exists and it is not degenerated. 
One could ask if our main result holds without (\textbf{A4}) but for the time being we cannot address this demanding question. Let us remark, however, that in the case of a simple random walk (for which (\textbf{A4}) is obviously violated) in \cite{Asselah_Zd} the authors established central limit theorem with normal distribution in the limit and in their proof they employ an analogous idea (so called no double backtracks at even times) to obtain non-degeneracy of the limit law. Unfortunately, this idea is designed for simple random walks and it is not clear whether it can be modified for more general walks (for instance for stable random walks).

A natural way to construct a random walk that satisfies our assumptions (\textbf{A1})-(\textbf{A4}) is to employ a recently introduced method of discrete subordination (see \cite{BSC}). To be more precise, let us consider the simple symmetric  random walk in $\mathbb{Z}^d$ that we denote by $\{Z_n\}_{n\geq 0}$. Further, let $\{\eta_n\}_{n\geq 0}$ be an increasing random walk in $\mathbb{Z}$ starting from $0$ that is independent of $\{Z_n\}_{n\geq 0}$, and which is uniquely determined by the following relation
\begin{align*}
\mathbb{E}[e^{-\lambda \eta_1}]=1-\psi (1-e^{-\lambda}).
\end{align*}
Here $\psi(\lambda)$ is a Bernstein function (see \cite{BFs_book}) such that $\psi (0)=0$ and $\psi (1)=1$. We then define the subordinate random walk as $S_n=Z_{\eta_n} $, $n\geq 0$. Such a random walk is aperiodic and symmetric. Moreover, it satisfies (\textbf{A3}) with index $0<\alpha\le2$ if and only if the function $\psi(\lambda)$ is regularly varying at zero with index $\alpha/2$ (see \cite{Mimica} and \cite{BCT}). For instance, one can take $\psi(\lambda)=\lambda ^{\alpha /2}$.
More general examples of random walks satisfying assumption (\textbf{A3}) may be found in \cite{williamson}.

We now state the main result of the article.

\begin{theorem}\label{CLT} Assume (\textbf{A1})-(\textbf{A4}) and $d /\alpha > 5 /2$. Then, there is a constant $\sigma_d>0$ such that
	\begin{equation}\label{eq:CLT}
	\frac{\calC_n - \bbE[\calC_n]}{\sqrt{n}}\xrightarrow[n\nearrow\infty]{\rm{(d)}}  \sigma_d\, \calN(0, 1),
	\end{equation}
	where $\calN(0, 1)$ stands for the standard normal distribution.
\end{theorem}

\subsection*{Outline of the proof}
Let us briefly explain the main steps of the proof. We follow the path of \cite{Asselah_Zd} but with a number of different ideas and approaches. The proof  itself follows from the Lindeberg-Feller central limit theorem \cite[Theorem 3.4.5]{Durrett} which requires a certain control of the asymptotic behavior (arithmetic mean and tail behavior) of   $\{\Var(\calC_n)\}_{n\ge0}$. As the key  results in this direction we show that
\begin{itemize}
\item[(i) ] the sequence $\{\Var(\calC_n)/n\}_{n\ge1}$ converges (see Lemma \ref{lm:existence_of_Varn/n_limit}), and
\item[(ii)] the limit is strictly positive (see Lemma \ref{positivity}). 
\end{itemize}
With this in hands and a  more general form of the following capacity decomposition 
\begin{equation*}
\calC_m+{\rm Cap}(\calR[m,n])-2G(\calR_m,\calR[m,n])\le\calC_n
\le 
\calC_m+{\rm Cap}(\calR[m,n]),\qquad 0\leq m\leq n,
\end{equation*}
which was obtained in \cite[Corollary 2.1]{Asselah_Zd}, we conclude that the left hand side in \eqref{eq:CLT} converges in distribution to a zero-mean normal law with variance $\sigma_d^2$ which is exactly the limit of $\{\Var(\calC_n)/n\}_{n\ge1}$. Here
\begin{align}\label{cap_dec}
G(\calR_m,\calR[m,n ])=\! \! \! \! \sum_{x\in \calR_m,\, y\in \calR[m,n]}G(x,y),\qquad 0\leq m\leq n,
\end{align}
 is the error term which is the main object to be studied in order to get estimates of the sequence $\{\Var(\calC_n)\}_{n\ge0}$. 
 
 The proof of step (i) follows the approach from \cite{Asselah_Zd} which bases on the estimates of the moments of $\{\calC_n\}_{n\ge0}$ extracted from \cite{LeGall-French}, combined with an application of Hammersley's lemma (see \cite{Hammersley}). Also, this is the place in the article where the restriction to $d/\alpha > 5 /2$ plays a key role.  
 
 To conclude step (ii) we require (\textbf{A4}). The proof  is based on a random time-change argument and  loop decomposition technique (see Subsection \ref{LOOPS}).

\subsection*{Literature overview and related results}

The study on the range process $\{\calR_n\}_{n\ge0}$ of a $\ZZ^d$-valued random walk $\{S_n\}_{n\ge0}$ has a long history. A pioneering work is due to Dvoretzky and Erd\"{o}s \cite{Dvoretzky} where they obtained a law of large numbers for $\{\#\calR_n\}_{n\ge0}$ when $\{S_n\}_{n\ge0}$ is the simple random walk and $d\ge2$. Here, $\#\calR_n$ denotes the cardinality of $\calR_n$. The result was later extended by Spitzer  \cite{Spitzer} for an arbitrary random walk in $d\ge1$.  Central limit theorem for $\{\#\calR_n\}_{n\ge0}$ was obtained by Jain and Orey \cite{Jain-Orey} when $\{S_n\}_{n\ge0}$ is strongly transient. 
Le Gall and Rosen \cite{LeGall-Rosen} were the first who considered the strong law of large numbers and the central limit theorem for 
$\{\#\calR_n\}_{n\ge0}$  in the case when $\{S_n\}_{n\ge0}$ is a stable aperiodic random walk, that is, when it satisfies (\textbf{A1}) and (\textbf{A3}). On the other hand, the first results on the long-time behavior  of the capacity process $\{\calC_n\}_{n\ge0}$  are due to  Jain and Orey \cite{Jain-Orey} where they obtained a version of the strong law of large numbers for any transient random walk. 
Very recently Asselah, Schapira and Sousi \cite{Asselah_Zd} proved  a central limit theorem  for $\{\calC_n\}_{n\ge0}$ for the simple random walk in $d\ge6$. In the case when $d=5$ Schapira \cite{Schapira} obtained analogous result for a class of symmetric random walks which fulfil some moment condition. Versions of  a law of large numbers and central limit theorem in the case $d=4$ 
 were proved by Asselah, Schapira and Sousi in \cite{Asselah_Z4}, 
 see also \cite{Chang} for $d=3$.
 Asselah and  Schapira  \cite{Asselah} established also a 
 large deviation principle for $d\ge5$.
 
 The aim of this article is to obtain a central limit theorem for the capacity  of the range process for a class of $\alpha$-stable strongly transient random walks with the index satisfying $d /\alpha> 5 /2$. To the best of our knowledge this is the first result in this direction dealing with random walks that do not have finite second moment. Our  motivation comes from the  article by Le Gall and Rosen \cite{LeGall-Rosen}, and  approach developed by Asselah, Schapira and Sousi \cite{Asselah_Zd}. A type of the limit behaviour of the sequence $\{\calC_n\}_{n\ge0}$ depends on the value of the ratio $d/\alpha$. 
We observe that our central limit theorem reveals that the capacity of the range of stable random walks with $d/\alpha>5 /2$ behaves as the cardinality of the range for $d/\alpha>3 /2$, cf. \cite[Result 1]{LeGall-Rosen}. 
If $d/\alpha = 5/2$ we conjecture that the limit law is again normal but the scaling sequence should be of the form $\sqrt{ng(n)}$, where  $g(n)= \sum_{k=1}^n k^2b(k)^{-2d}$ is a slowly varying function. This corresponds to the scaling sequence for the range process in the case $d/\alpha = 3/2$, as established in \cite[Section 4.5]{LeGall-Rosen}. We remark that the case when $\alpha =2$ and $d=5$ has been recently partially solved by Schapira in \cite{Schapira}, where the author considers a class of symmetric random walks which satisfy certain moment condition and  obtains the normal law in the limit under the scaling $\sqrt{n\log n}$.
For $2\leq d/\alpha<5/2$ the limit law should be non-normal and we expect yet another scaling sequence which would involve the truncated Green function, cf.\ \cite[Result 2]{LeGall-Rosen}. 
%
%
The study of the case  $2 \le d/\alpha \leq 5 /2$ is an ongoing project and it is postponed to follow-up articles.

\section{On the SLLN for $\{\calC_n\}_{n\ge0}$}

In this section, we prove that under  (\textbf{A2}) the sequence $\{\calC_n\}_{n\ge0}$ satisfies a (version of) strong law of large numbers with strictly positive limit. This result will be  crucial  in showing that the limit in \eqref{CLT} is non-degenerate (see Section \ref{S5}).
Recall first that 
for any transient random walk on $\ZZ^d$ it holds that the corresponding capacity process 
 $\{\calC_n\}_{n\ge0}$ satisfies \begin{equation}\label{Jain_Orey}
\lim_{n\nearrow \infty}\frac{\calC_n}{n}= \mu_{d}\qquad \bbP\text{-a.s}.
\end{equation} (see \cite[Theorem 2]{Jain-Orey}). In the rest of this section we show that under  (\textbf{A2})
the constant $\mu_d$ is necessarily strictly positive. We start with the following auxiliary lemma. Recall that for $A,B\subseteq\ZZ^d$ the quantity $G(A,B)$ is defined as 
\begin{align*}
G(A,B) =\sum_{x\in A,\,y\in B}G(x,y).
\end{align*}

\begin{lemma}\label{lm:exp_g_strong_transience}
	Assume  (\textbf{A2}). Then there is a constant $C>0$ such that
	\begin{align*}
	\bbE\left[ G(\calR_n, \calR_n)\right] \le Cn,\qquad n\geq 1.
	\end{align*}
\end{lemma}
\begin{proof}
	 We have
	\begin{align*}
	\bbE\left[ G(\calR_n, \calR_n)\right] 
	&=\bbE\ugl{\sum_{k, l = 0}^n G(S_k, S_l)}
	= \bbE\ugl{\sum_{k = 0}^n G(S_0) 
		+ \sum_{k = 0}^n \sum_{\substack{l = 0 \\ l \neq k}}^n G(S_{\aps{k - l}})} \\
	& \le \bbE\ugl{(n+1) G(0) + 2\sum_{k = 0}^n  \sum_{l = 1}^n G(S_l)} \le 2 G(0)(n+1) \obl{1 + \sum_{l = 1}^n \bbE[ G(S_l)]} \\
	& = 2G(0)(n+1)\obl{1 + \sum_{l = 1}^n \sum_{x \in \bbZ^d} G(x) \bbP(S_l = x)} \\
	& = 2G(0)(n+1) \obl{1 + \sum_{l = 1}^n \sum_{x \in \bbZ^d} \sum_{k = 0}^{\infty} p_k(0, x) p_l(x, 0)} \\
	& = 2G(0)(n+1)\obl{1 + \sum_{l = 1}^n \sum_{k = 0}^{\infty} p_{k + l}(0)}\\
	&\le 2G(0)(n+1)\obl{1 + \sum_{l = 1}^{\infty} \sum_{k = l}^{\infty} p_k(0)}\\
	& = 2G(0)(n+1) \obl{1 + \sum_{k = 1}^{\infty} kp_k(0)} \le Cn,
	\end{align*}		
	where the last inequality follows from (\textbf{A2}).
\end{proof}
We now show that $0$ cannot be an accumulation point of $\{\mathbb{E}[\calC_n]/n\}_{n\ge1}$.
\begin{proposition}\label{prop:E[C_n]_ge_cn}
	Assume  (\textbf{A2}). Then there is a constant $c > 0$ such that 
	\begin{align*}
	\liminf_{n\nearrow\infty}\frac{\bbE[\calC_n]}{n} \ge c.
	\end{align*}
\end{proposition}
\begin{proof}

	For fixed $n\ge1$ we consider the following (random) probability measure defined on $\bbZ^d$ 
	\begin{equation}\label{eq:def_of_nu}
	\nu_n(x) = \frac{1}{n} \sum_{k = 1}^{n} \delta_{S_k}(x).
	\end{equation}
	Clearly, $\mathrm{supp}\,\nu_n =\mathcal{R}[1,n]$.
	 According to \cite[Lemma 2.3]{Jain}, for symmetric random walks the capacity of a set $A\subseteq\ZZ^d$ has the following representation 
	 $$\mathrm{Cap}(A)=\frac{1}{\inf_\nu\sum_{x,y\in A}G(x,y)\nu(x)\nu(y)},$$ 
	 where the infimum is taken over all probability measures on $\ZZ^d$ with $\mathrm{supp}\,\nu\subseteq A.$
	By setting \begin{align*}
	\calJ(\nu_n)
	= \sum_{x, y \in \mathcal{R}_n} G(x, y) \nu_n(x)\nu_n(y) 
	= \frac{1}{n^2} G(\mathcal{R}[1,n], \mathcal{R}[1,n]),
	\end{align*}
 we obtain
	$
	\calC_n \ge (\calJ(\nu_n))^{-1}. 
	$
	Finally, by Jensen's inequality we have that	\begin{equation*}\label{eq:E[Cap(Q_R)]_ge}
	\bbE[\mathcal{C}_n] \ge \big(\bbE[\calJ(\nu_n)]\big)^{-1},
	\end{equation*} which together with
 Lemma \ref{lm:exp_g_strong_transience} proves the assertion.
\end{proof}
As a direct consequence of \eqref{Jain_Orey} and Proposition \ref{prop:E[C_n]_ge_cn} we conclude strict positivity of the constant $\mu_d$.
\begin{corollary}
	Under (\textbf{A2}) it holds that $\mu_d>0$.
\end{corollary}

\section{Error term estimates}

The goal of this section is to obtain estimates of the error term which is of the form  \eqref{cap_dec}. This will be crucial in the analysis of the sequence $\{\Var(\calC_n)\}_{n\ge0}$.
In the sequel we assume (\textbf{A3}). Recall that the function $b(x)$ is necessarily of the following form
\begin{equation*}
b(x) = x^{1/\alpha} \ell(x),\qquad x\geq 0,
\end{equation*}
where $\alpha \in (0,2]$ and $\ell(x)$ is a slowly varying function. Without loss of generality we may assume that $b(x)$ is continuous, increasing and $b(0)=0$ (see \cite{BGT_book}). If, in addition, (\textbf{A1}) holds true, then 
 by \cite[Proposition 2.4.]{LeGall-Rosen} there exists a constant $C > 0$ such that for any $n \ge 0$ and $x \in \bbZ^d$,
\begin{equation}\label{eq:bound_on_p(x)}
p_n(x) \le C(b(n))^{-d}.
\end{equation}
Recall that $\{S_n\}_{n\ge0}$ is transient if $d/ \alpha>1$ and it is strongly transient if $d/\alpha>2$.  Further, for $n\ge0$
we write $G_n(x, y)$ for the Green function up to time $n$, that is,
\begin{equation*}
G_n(x, y) = \sum_{k = 0}^n p_k(x,y),\qquad x,y\in\ZZ^d.
\end{equation*}
Also, similarly as before, we use the notation $G_n(x)= G_n(0,x)$, $x\in\ZZ^d$.
We start with the following auxiliary lemma.
\begin{lemma}\label{lm:bounds_on_GGG} Assume (\textbf{A1})-(\textbf{A3}). Then
	there exists a constant $C > 0$ such that for all $n\geq 1$ and  all $a \in \bbZ^d$,
	\begin{equation}\label{eq:bound_with_h_d}
	\sum_{x,y\in\ZZ^d} G_n( x) G_n( y) G(y-x + a) \le C h_d(n),
	\end{equation}
	where $h_d(n)$ is given by
	\[   
	h_d(n) = 
	\begin{cases}
	1, &\quad d/\alpha  > 3,\\
	\sum_{k=1}^n k^{-1}\ell (k)^{-d}, & \quad d /\alpha = 3,\\
	n^3 (b(n))^{-d}, & \quad 2< d/\alpha < 3.\\
	\end{cases}
	\]
	Observe that the function $n\mapsto \sum_{k=1}^n k^{-1}\ell (k)^{-d}$ is non-decreasing and slowly varying.
\end{lemma}
\begin{proof}
	By \eqref{eq:bound_on_p(x)} we have that for all $k, j \geq 0$,
	\begin{align*}
	\sum_{x,y\in\ZZ^d } 
	p_k(0, x) p_j(0, y)  G(x, y + a) &= \sum_{i = 0}^{\infty} \sum_{x, y\in\ZZ^d } p_k(0, x) p_j(a, y + a) p_i(x, y + a) \\
	& = \sum_{i = 0}^{\infty} \sum_{x, y \in\ZZ^d} p_k(0, x) p_i(x, y + a) p_j(y + a, a) \\
	&= \sum_{i = 0}^{\infty} p_{k + i + j}(0, a)
	\le c_1\sum_{i = k + j}^{\infty} b(i)^{-d} \\
	&\le c_2 (j + k)\, b(j + k)^{-d},
	\end{align*}
	where the last inequality follows from \cite[Proposition 1.5.10]{BGT_book}. Summing over $j$ from the set $\{0,1,\ldots ,n\}$ yields
	\begin{align*}
	\sum_{x,y\in\ZZ^d  } 
	p_k(0, x) G_n( y) G(x, y+a) 
	\le c_2 \sum_{j = 0}^n (j + k) b(j + k)^{-d} 
	\le c_3 k^2 b(k)^{-d},
	\end{align*}
	where we again used \cite[Proposition 1.5.10]{BGT_book} together with the fact that $d/\alpha > 2$. Summing over $k=0,1,\ldots ,n$ gives 
	\begin{equation*}
	\sum_{x,y\in\ZZ^d }  G_n( x) G_n( y) G(x, y+a) \le c_3 \sum_{k = 0}^n k^2 b(k)^{-d}.
	\end{equation*}
	For $d/\alpha > 3$ we can again apply \cite[Proposition 1.5.10]{BGT_book} to get
	\begin{equation*}
	\sum_{k = 1}^{\infty} k^2 b(k)^{-d} < \infty.
	\end{equation*}
	Hence, for $d /\alpha> 3$ we set $h_d(x) = 1$. 
	If $d /\alpha= 3$ we obtain
	\begin{equation*}
	\sum_{k = 1}^n k^2 b(k)^{-d} = \sum_{k = 1}^n k^{-1}\ell (k)^{-d},
	\end{equation*}
	as desired. We mention that slow variation of $n\mapsto \sum_{k=1}^n k^{-1}\ell (k)^{-d}$ follows  from \cite[Lemma 2.2]{LeGall-Rosen}.
	Finally, for $2 < d/\alpha < 3$ we apply \cite[Theorem 1.5.11]{BGT_book} to get
	\begin{equation*}
	\sum_{k = 1}^n k^2 b(k)^{-d} \le c_4 n^3 b(n)^{-d},\qquad n\geq 1,
	\end{equation*}
	what finishes the proof.
\end{proof}

We next obtain estimates of the error term. Let us remark here that a similar result has been obtained in \cite[Lemma 3.2]{Asselah_Zd} for the simple random walk only.  We give an alternative proof of this result which relies on the Markovian structure of random walks and was motivated by techniques that were applied to estimate moments of intersection times for random walks, cf. \cite{Chen} and \cite{Lawler}.
Our approach is valid for all random walks satisfying  (\textbf{A1})-(\textbf{A3}).

\begin{lemma}\label{lm:E_of_sum_sum_G^k} Assume (\textbf{A1})-(\textbf{A3}). Let $\{S'_n\}_{n\ge0}$ be an independent copy of $\{S_n\}_{n\ge0}$ and denote the corresponding range process by $\{\calR'_n\}_{n\ge0}$.
Then, for all $k, n \ge1$ we have that
	\begin{equation*}
	\bbE\left[\big(G(\calR_n, \calR^\prime _n)\big)^k\right] \le C h_d(n)^k,
	\end{equation*}
	where $C>0$ is a constant that depends on $k$, and
	$h_d(n)$ is defined in Lemma \ref{lm:bounds_on_GGG}.
\end{lemma}
\begin{proof}
	Let us consider the hitting times $T_x= \inf\{ n\geq 0:\, S_n=x\}$, $x\in\ZZ^d$. It then holds
	\begin{align*}
	\bbE\left[G(\calR_n, \calR^\prime _n)\right] = \sum_{x,y\in\ZZ^d} \mathbb{P} (T_x\leq n)\mathbb{P} (T_y\leq n)G(x,y).
	\end{align*}
	Since $\mathbb{P} (T_x\leq n)\leq G_n(x)$, for $k=1$ we conclude the result in view of Lemma \ref{lm:bounds_on_GGG}. For $k>1$ we proceed as follows. We first observe that
	\begin{align*}
	\bbE\left[\big(G(\calR_n, \calR^\prime _n)\big)^k\right] 
	= 
	\sum_{x_1\dots,x_k\in\ZZ^d}\sum_{y_1\dots ,y_k\in\ZZ^d}
	\bbE\Big[ \prod_{i=1}^k \bbjedan_{\{x_i\in \calR_n\}}\Big]
	\bbE\Big[ \prod_{i=1}^k \bbjedan_{\{y_i\in \calR_n\}}\Big]
	\prod_{i=1}^k G(x_i, y_i).
	\end{align*}
	For simplicity we use notation 
\begin{align*}
r_n(x_1,\ldots ,x_j) = \mathbb{P}(T_{x_1}\leq \ldots \leq T_{x_j}\leq n),\qquad 
j\ge1,\ x_1,\dots,x_j\in\ZZ^d.
\end{align*}	
	We clearly have
	\begin{align*}
	\bbE\Big[ \prod_{i=1}^k \bbjedan_{\{x_i\in \calR_n\}}\Big] \leq
	\sum_{\pi \in \Pi(k)} r_n(x_{\pi (1)},\ldots ,x_{\pi(k)}),
	\end{align*}
	where $\Pi(k)$ is the set of all permutations of the set $\{1,\ldots ,k\}$. Hence
	\begin{align*}
	\sum_{x_1,\dots ,x_k\in\ZZ^d} \bbE\Big[ \prod_{i=1}^k \bbjedan_{\{x_i\in \calR_n\}}\Big] \leq
	k! \!\! \sum_{x_1,\dots ,x_k\in\ZZ^d} r_n(x_1,\ldots ,x_k). 
	\end{align*}
	Notice that the strong Markov property employed at time $T_{x_{k-1}}$ implies that
	\begin{align*}
	r_n(x_1,\ldots ,x_k) \leq 
	r_n(x_1,\ldots ,x_{k-1})\, \mathbb{P}_{x_{k-1}}(T_{x_k}\leq n).
	\end{align*} 
	We thus obtain
	\begin{align*}
	\bbE\left[\big(G(\calR_n, \calR^\prime _n)\big)^k\right] 
	&\leq
	(k!)^2 \!\!\!\! \sum_{x_1\dots,x_{k-1}\in\ZZ^d}\sum_{y_1\dots ,y_{k-1}\in\ZZ^d}
	\!\!\!\!\!\!
	r_n(x_1,\ldots ,x_{k-1})\, 
	r_n(y_1,\ldots ,y_{k-1})
	\prod_{i=1}^{k-1} G(x_i, y_i)\\
	&\qquad \qquad \qquad 
	\times \sum_{x_k,y_k\in\ZZ^d}
	\mathbb{P}_{x_{k-1}}(T_{x_k}\leq n)
	\mathbb{P}_{y_{k-1}}(T_{y_k}\leq n)
	G(x_k, y_k).
	\end{align*}
	For the last term we have
	\begin{align*}
	\sum_{x_k,y_k\in\ZZ^d}
	\mathbb{P}_{x_{k-1}}(T_{x_k}\leq n)
	\mathbb{P}_{y_{k-1}}(T_{y_k}\leq n)
	G(x_k, y_k)
	\leq 
	\sum_{x_k,y_k\in\ZZ^d} G_n(x_{k-1}, x_k)G_n(y_{k-1}, y_k)	G(x_k,y_k)	
	\end{align*}
	and, by Lemma \ref{lm:bounds_on_GGG}, the last sum is bounded by a constant times $h_d(n)$. By repeating the same argument $k$ times we get the result.
\end{proof}

\section{Variance estimates}

In this section, we show that the limit of the sequence $\{\Var(\calC_n)/n\}_{n\ge1}$  exists if $d/ \alpha>5 /2$.  We follow the approach from \cite[Lemma 3.5]{Asselah_Zd}. The proof is based on the following two results which we state  for reader's convenience. The first one is Hammersley's lemma.

\begin{lemma}[{\cite[Theorem 2]{Hammersley}}]\label{Hammersley}
	Let $\{a_n\}_{n\ge1}$ and $\{b_n\}_{n\ge1}$  be sequences of real numbers satisfying $$ a_{n+m}\le a_n+a_m+b_{n+m},\qquad n,m\ge1.$$ If $\{b_n\}_{n\ge1}$ is non-decreasing and $$\sum_{n=1}^\infty\frac{b_n}{n^2}<\infty,$$ then $\{a_n/n\}_{n\ge1}$ converges to a finite limit.
	\end{lemma}

The second  is the capacity decomposition formula discussed in the introduction.

\begin{lemma}[{\cite[Proposition 25.11]{Spitzer} and \cite[Proposition 1.2]{Asselah_Zd}}]\label{capacity_decomp} 
Let $A,B\subset\ZZ^d$ be finite. Then $$\mathrm{Cap}(A)+\mathrm{Cap}(B)-2G(A,B)\le\mathrm{Cap}(A\cup B)\le \mathrm{Cap}(A)+\mathrm{Cap}(B)-\mathrm{Cap}(A\cap B).$$
	\end{lemma}
We remark that Proposition 1.2 in \cite{Asselah_Zd} is stated for a simple random walk only, but the proof is valid for an arbitrary transient random walk.

\begin{lemma}\label{lm:existence_of_Varn/n_limit}
	Assume (\textbf{A1})-(\textbf{A3}) and 
	 $d/\alpha > 5 /2$. Then the sequence $\{\Var(\calC_n)/n\}_{n\ge1}$ converges to $\sigma_d^2\ge0.$
\end{lemma}
\begin{proof}
		Let $n,m\ge1$ be arbitrary. Due to space homogeneity of the capacity, that is, $\mathrm{Cap}(A)=\mathrm{Cap}(A+x),$ $x\in\ZZ^d$, $A\subseteq\ZZ^d$, we have that $$\calC_{n+m}=\mathrm{Cap}(\calR_{n+m}-S_n)=\mathrm{Cap}(\{S_0-S_n,\dots,S_n-S_n\}\cup\{S_{n}-S_n,\dots,S_{n+m}-S_n\}).$$ 
		Thus, according to  Lemma \ref{capacity_decomp},
		\begin{equation}\label{decomp}
	\calC_n^{(1)}+\calC_m^{(2)}-2G(\calR_n^{(1)},\calR_m^{(2)})\le \calC_{n+m}\le \calC_n^{(1)}+\calC_m^{(2)},\end{equation} where $\calC_n^{(1)}$ and $\calC_m^{(2)}$ ($\calR_n^{(1)}$ and $\calR_m^{(2)}$) are independent and have the same law as $\calC_n$ and $\calC_m$ ($\calR_n$ and $\calR_m$), respectively.
	Further, for $k\ge1$ define	 $\overline{\calC}_k = \calC_k - \bbE[\calC_k]$, and similarly  $\overline{\calC}_k^{(1)}$ and $\overline{\calC}_k^{(2)}$. Taking expectation in \eqref{decomp} and then subtracting those two relations yields
	\begin{equation*}
	\big\vert \overline{\calC}_{n+m} - (\overline{\calC}_n^{(1)} + \overline{\calC}_m^{(2)})\big\vert \le 2 \max \{G(\calR_n^{(1)},\calR_m^{(2)}), \bbE[G(\calR_n^{(1)},\calR_m^{(2)})]\}.
	\end{equation*}
	Denote  $\lVert\cdot\rVert_2=\bbE[(\cdot)^2]^{1/2}$. Clearly, $\Var(\calC_k)=\lVert\overline{\calC}_k\rVert_2^2$, $k\ge1$.
The triangle inequality and independence of $\overline{\calC}_n^{(1)}$ and $\overline{\calC}_n^{(2)}$ together with  the estimate $\bbE[G(\calR_n^{(1)},\calR_m^{(2)})] \le \Vert G(\calR_n^{(1)},\calR_m^{(2)})\Vert_2$ and Lemma \ref{lm:E_of_sum_sum_G^k} imply
	\begin{align*}
	\Vert\overline{\calC}_{n+m}\Vert_2
	& \le \big(\Vert\overline{\calC}_n^{(1)}\Vert_2^2 + \Vert\overline{\calC}_m^{(2)}\Vert_2^2\big)^{1/2} + 4\Vert G(\calR_n^{(1)},\calR_m^{(2)})\Vert_2 \\
	&\le \big(\Vert\overline{\calC}_n^{(1)}\Vert_2^2 + \Vert\overline{\calC}_m^{(2)}\Vert_2^2\big)^{1/2} + c_1 h_d(n+m),
	\end{align*}
	where in the last inequality we used $$G(\calR_n^{(1)},\calR_m^{(2)})\le G(\calR_{n+m}^{(1)},\calR_{n+m}^{(2)}).$$ 
	Consequently,  
	\begin{align*}
	\Vert\overline{\calC}_{n + m}\Vert_2^2
	\le \Vert\overline{\calC}_n\Vert_2^2 + \Vert\overline{\calC}_m\Vert_2^2 + c_2 \big(\Vert\overline{\calC}_n\Vert_2^2 + \Vert\overline{\calC}_m\Vert_2^2\big)^{1/2} h_d(n+m) + c_3h^2_d(n+m). 
	\end{align*}
	By setting $a_k=\Vert\overline{\calC}_{k}\Vert_2^2$, $k\ge1$,  
	the above relation reads 
	\begin{equation}\label{to show_1}
	a_{n+m}\le a_n+a_m+c_2 \big(\Vert\overline{\calC}_n\Vert_2^2 + \Vert\overline{\calC}_m\Vert_2^2\big)^{1/2} h_d(n+m) + c_3h_d^2(n+m).
		\end{equation}
		In the sequel we find an upper bound for the third term of the right hand side of inequality \eqref{to show_1}.  
		
We  first consider the case $d/ \alpha\geq 3$ for which $h_d(n)$ is slowly varying. 
If we  prove that
\begin{align*}
\Vert \overline{\calC}_n\Vert_2\le c\sqrt{n}\,h_d(n),\qquad n\ge1,
\end{align*}	
then by defining $b_k=c\sqrt{k}h^2_d(k)$, $k\ge1$, the assertion of the lemma will follow directly from \eqref{to show_1} and Lemma \ref{Hammersley}.
For  $k \ge 1$  we set
	\begin{equation}\label{eq:def_of_a_k}
	\alpha_k = \sup \left\{ \Vert \overline{\calC}_i\Vert_2:\, 2^k \le i \le 2^{k + 1}\right\}.
	\end{equation}
	Further, for $k \ge 2$ we take $n\ge1$ such that $2^k \le n < 2^{k + 1}$, and we set $l = \floor{n/2}$ and $m = n - l$. Here $\floor {a}$ stands for the largest integer smaller than or equal to $a\in\R$. Analogously as above we have
	\begin{equation*}\label{eq:Cor_2.1.-existance_of_limit}
	\calC_l^{(1)} + \calC_m^{(2)}  -2G(\calR_l^{(1)},\calR_m^{(2)}) \le \calC_n \le \calC_l^{(1)} + \calC_m^{(2)},
	\end{equation*}
	and we arrive at
	\begin{align}\label{eq:inter}
	\Vert\overline{\calC}_n\Vert_2
	&  \le \big(\Vert\overline{\calC}_l^{(1)}\Vert_2^2 + \Vert\overline{\calC}_m^{(2)}\Vert_2^2\big)^{1/2} + c_4 h_d(n).
	\end{align}
	 Recall that $\calC_l^{(1)}$ and $\calC_m^{(2)}$ ($\calR_l^{(1)}$ and $\calR_m^{(2)}$) are independent and have the same law as $\calC_l$ and $\calC_m$ ($\calR_l$ and $\calR_m$), respectively.
	Hence, equation \eqref{eq:inter} implies
	\begin{equation*}
	\Vert\overline{\calC}_n\Vert_2 \le 2^{1/2} \alpha_{k - 1} + c_4 h_d(n).
	\end{equation*}
	Taking supremum over $2^k \le n \le 2^{k + 1}$ yields	\begin{equation*}\label{eq:rec_relation_for_a_k}
	\alpha_k \le 2^{1/2} \alpha_{k - 1} + c_4 h_d(2^{k + 1}).
	\end{equation*}
	We next set $\beta_k = \alpha_k / h_d(2^k)$. In view of 
	\cite[Theorem 1.5.6]{BGT_book} we get that $h_d(2^{k + 1}) \le ch_d(2^k)$. This and monotonicity of $h_d(n)$ gives 
	\begin{equation*}
	\beta_k \le 2^{1/2} \beta_{k - 1} + c_5.
	\end{equation*}
	By iteration of this inequality we have $\beta_k \le c_6 2^{k/2}$ which  implies $\alpha_k \le c_6 2^{k/2} h_d(2^k)$. Finally, using definition of $\alpha_k$ and \cite[Theorem 1.5.6]{BGT_book} we obtain
$$
\lVert\overline\calC_n\rVert_2^2 \le \alpha_k^2 \le c_6^2 2^k h^2_d(2^k) \le c_7 n\, h^2_d(n),
	$$
which finishes the proof for $d/\alpha\geq 3$.

Next we consider the case $5 /2<d/\alpha<3$. We set $\Delta = d/\alpha -5/2$ and we observe that 
\begin{align*}
h_d(n) = n^{1/2-\Delta}s(n),
\end{align*}
where $s(n)=(\ell (n))^{-d}$ is a slowly varying function. If we prove
\begin{align}\label{to_show_2}
\Vert \overline{\calC}_n\Vert_2\le c n^{3\Delta /2}\,h_d(n),\quad n\ge1.
\end{align}
then by defining $b_k=c k^{3\Delta /2} h^2_d(k)$, $k\ge1$ the assertion of the lemma will again follow directly from \eqref{to show_1} and Lemma \ref{Hammersley}. 
As $h_d(n)$ is regularly varying of index $1/2-\Delta$, by \cite[Theorem 1.5.6]{BGT_book}, we have that there is $k_0\geq 1$ such that
$h_d(2^{k-1})/h_d(2^k)\leq 2^{(3\Delta -1)/2 }$ for $k\geq k_0$. 
We set again $\beta_k = \alpha_k / h_d(2^k)$.
Dividing \eqref{eq:inter} by $h_d(2^k)$ and using the fact that $h_d(n)\leq c_8 h_d(2^k)$ implies that for $k\geq k_0$
\begin{align}\label{for_k_0}
\beta_k \leq 2^{3\Delta /2}\beta_{k-1} + c_{9}.
\end{align}
By monotonicity of $h_d(n)$, there exists $M\geq 1$ such that $h_d(2^{k-1})/h_d(2^k)\leq M 2^{(3\Delta -1)/2}$, $k\geq 1$, and thus we may easily extend \eqref{for_k_0} to all $k\geq 1$. Hence by iterating \eqref{for_k_0} we  obtain
\begin{align*}
\alpha_k \leq c_{10} 2^{3\Delta\, k /2}h_d(2^k),\quad k\geq 1,
\end{align*}
with $\alpha_k $ is defined in \eqref{eq:def_of_a_k}.
We thus conclude \eqref{to_show_2} and the result follows.
\end{proof}

\section{CLT for $\{\calC_n\}_{n\ge0}$}\label{S5}

In this section, we first show
strict positivity of the limit $\sigma_d^2$ from Lemma \ref{lm:existence_of_Varn/n_limit}, and then we finally   prove  Theorem \ref{CLT}. Namely, $\sigma_d^2$ will be exactly  the variance  parameter  of the limiting normal law in  \eqref{eq:CLT}. 
To show that   $\sigma_d^2$ is strictly positive we adapt an idea from \cite{Asselah_Zd} where the simple random walk is decomposed into two independent processes. The first process is the process counting the number of double-backtracks, and the second one is the process with no double-backtracks. 
For our class of random walks we use one-step loops 
 instead of double-backtracks. To be more precise, 
we say that  $\{S_n\}_{n\ge0}$ makes a one-step loop at time $n$ if $S_n = S_{n - 1}$. Clearly, $\{S_n\}_{n\ge0}$ admits one-step loops if and only if $p_1(0)>0$. Also, when the walk makes a one-step loop, the range evidently remains unchanged.
We will first build a random walk $\{\widetilde S_n\}_{n\ge0}$ with no one-step loops, and then we will show how to construct a random walk $\{\widehat S_n\}_{n\ge0}$ starting from  $\{\widetilde S_n\}_{n\ge0}$ with (i) the same law as $\{S_n\}_{n\ge0}$, and (ii) the range process being a certain random time-change of the range process of $\{\widetilde S_n\}_{n\ge0}$.

Finally, to prove Theorem \ref{CLT} we combine Lindeberg-Feller central limit theorem (see \cite[Theorem 3.4.5]{Durrett} or Lemma \ref{feller}) and a dyadic version of the capacity decomposition formula from Lemma \ref{capacity_decomp} (see \cite[Corollary 2.1]{Asselah_Zd} or Lemma \ref{dyadic}).

\subsection*{Strict positivity of $\sigma_d^2$} \label{LOOPS}
 Assume (\textbf{A1})-(\textbf{A4}), and
let  $\{\widetilde{X}_i\}_{i\in\N}$ be a sequence of i.i.d. random variables with  distribution
\[   
\bbP(\widetilde{X}_i = x) = \frac{p_1(x)}{1 - p}\bbjedan_{\{x\neq 0\}}.
\]
Recall that $p=p_1(0)>0$ (assumption (\textbf{A4})).
Further, let $\{\widetilde{S}_n\}_{n\ge0}$ be the corresponding random walk. 
Clearly, $\{\widetilde{S}_n\}_{n\ge0}$ has no one-step loops.
We now construct a  random walk $\{\widehat{S}_n\}_{n\ge0}$ by adding an independent geometric number of one-step loops to $\{\widetilde{S}_n\}_{n\ge0}$ at each step, and which has the same law as $\{S_n\}_{n\ge0}$. Let $\{\xi_i\}_{i \ge 0}$ be a sequence of i.i.d. geometric random variables with parameter $p$ which are independent of $\{\widetilde{S}_n\}_{n\ge0}$. Recall, 
\begin{equation*}
\bbP(\xi_i = k) = p^k (1 - p), \qquad k \ge 0.
\end{equation*}
We set $N_{-1} = 0$,
\begin{equation*}
N_k = \sum_{i = 0}^k \xi_i, \qquad k\ge0,
\end{equation*}
and  define $\{\widehat{S}_n\}_{n\ge0}$ according to the following procedure: We start by setting $\widehat S_0 = 0$. Further, for  $k \ge 0$ we define
\begin{equation*}
I_k = [k + N_{k - 1} + 1, k + N_k].
\end{equation*} 
If $I_k \neq \emptyset$ then for each $i \in I_k$ we define $\widehat S_i = \widetilde{S}_k$. We next follow the path of $\{\widetilde{S}_n\}_{n\ge0}$ which means that we set $\widehat S_{k + N_k + 1} = \widetilde{S}_{k + 1}$. This construction provides a random walk $\{\widehat{S}_n\}_{n\ge0}$ with the same law as $\{S_n\}_{n\ge0}$. We also have
\begin{equation*}
\widetilde{S}_n = \widehat S_{n + N_n} = \widehat S_{n + N_{n - 1}},\qquad n\ge0,
\end{equation*} 
where the second equality holds since $n + N_{n - 1} = (n - 1) + N_{n - 1} + 1$. 
Consequently, 
\begin{equation*}
\widetilde{\calR}_n = \widehat\calR_{n + N_n} = \widehat\calR_{n + N_{n - 1}},\qquad n\ge0,
\end{equation*}
where $\{\widetilde{\calR}_n\}_{n\ge0}$ and $\{\widehat{\calR}_n\}_{n\ge0}$ are range processes of $\{\widetilde{S}_n\}_{n\ge0}$ and $\{\widehat{S}_n\}_{n\ge0}$, respectively.

We now show that $\sigma_d$ must be strictly positive.
We first establish two technical lemmas.
\begin{lemma}\label{lm:sum_sum_G_ge_sqrt_n}
	Assume (\textbf{A1})-(\textbf{A4}) and $d/\alpha> 5 /2$. Then for any  $c,c'>0$ it holds that
	\begin{equation*}
	\lim_{n\nearrow \infty}	\bbP \big( G\big(\widetilde{\calR}_n ,  \widetilde{\calR}[n, \lfloor n + c'n\rfloor] \big)   \ge c\sqrt{n} \big)
	=0,
	\end{equation*} 
	where $G(x,y)$ is the Green function of $\{S_n\}_{n\ge0}$ (or $\{\widehat{S}_n\}_{n\ge0}$).
\end{lemma}
\begin{proof}
	Let $M_n$ be the number of one-step loops added in the time interval $[n, n + c'n]$, that is,
	\begin{align*}
	M_n = \xi_n +\ldots + \xi_{\floor{n + c'n}}.
	\end{align*}
	Since $ \widetilde{S}_{\floor{n + c'n}} = \widehat S_{\floor{n + c'n} + N_{\floor{n + c'n}}}$ and
	\begin{equation*}
	N_{\floor{n + c'n}} = \sum_{i = 0}^{\floor{n + c'n}} \xi_i = \sum_{i = 0}^{n - 1} \xi_i + \sum_{i = n}^{n + c'n} \xi_i = N_{n - 1} + M_n,
	\end{equation*}
	we have $\widetilde{S}_{\lfloor n + c'n\rfloor} = \widehat S_{\lfloor n + c'n\rfloor + N_{n - 1} + M_n}$. This together with $\widetilde{S}_n = \widehat S_{n + N_{n - 1}}$, $n\ge0$, implies
	\begin{equation*}
	\widetilde{\calR}[n, \lfloor n + c'n\rfloor] = \widehat \calR[n + N_{n - 1}, \lfloor n + c'n\rfloor + N_{n - 1} + M_n].
	\end{equation*}
	Therefore
	\begin{align*}
	\bbP \big( &G\big(\widetilde{\calR}_n ,  \widetilde{\calR}[n, \lfloor n + c'n\rfloor] \big)   \ge c\sqrt{n} \big)\\
	&\leq \bbP \Big( G\big(\widehat \calR_{n + N_{n - 1}} , \, \widehat\calR[n + N_{n - 1}, \lfloor(1 + c' + c_1)n\rfloor + N_{n - 1}] \big)   \ge c\sqrt{n} \Big)
	+ \bbP(M_n \ge c_1n),
	\end{align*}
	where $c_1 > 0$ is a constant that we specify. For that, notice that there exists a constant $c_2 > 0$ such that $\bbE[M_n] \le c_2n$, $n\ge0$. Set $c_1 = c_2 + \varepsilon$ for some $\varepsilon > 0$. Then, by Chebyshev's inequality we have
	\begin{equation*}
	\bbP(M_n \ge c_1n) = \bbP(M_n - c_2n \ge \varepsilon n)  \le \frac{\Var(M_n)}{\varepsilon^2 n^2} \xrightarrow{n \nearrow \infty} 0.
	\end{equation*}
	To bound the first term of the penultimate estimate we observe that
 $G(x - a, y - a) = G(x, y)$, $x, y, a \in \bbZ^d$,  
	 and that the two random variables
\begin{align*}
\widehat\calR_{n + N_{n - 1}} - \widehat S_{n + N_{n - 1}}\quad \mathrm{and}\quad 
\widehat \calR[n + N_{n - 1}, \lfloor(1 + c' + c_1)n\rfloor + N_{n - 1}] - \widehat S_{n + N_{n - 1}}
\end{align*}	 
are independent. Thus, instead of the second random set we can write $\calR'_{\lfloor(c' + c_1)n\rfloor}$, where $\{\calR'_n\}_{n\ge0}$ is the range process of a random walk that is an independent copy of $\{\widehat{S}_n\}_{n\ge0}$. We obtain
	\begin{align*}
	\bbP \big( G\big(\widehat\calR_{n + N_{n - 1}}-\widehat S_{n + N_{n - 1}} &,  \calR'_{\lfloor(c' + c_1)n\rfloor}\big) \ge c\sqrt{n} \big)\\
	&\leq 
	\bbP \big( G\big(\widehat\calR_{\lfloor(1 + c_3)n\rfloor} ,  \calR'_{\lfloor(c' + c_1)n\rfloor} \big) \ge c\sqrt{n} \big)
	+ \bbP(N_{n - 1} \ge c_3 n),
	\end{align*}
	where the constant $c_3$ is defined as above to make $\bbP(N_{n - 1} \ge c_3 n)$ tending to zero as $n$ goes to infinity. We finally set $c_4 = \max\{\lfloor1 + c_3\rfloor, \lfloor c' + c_1\rfloor\}$ and we apply the Markov inequality, Lemma \ref{lm:E_of_sum_sum_G^k} and \cite[Theorem 1.5.6]{BGT_book} to get
	\begin{align*}
	\bbP \big( G\big(\widehat\calR_{\lfloor(1 + c_3)n\rfloor} ,  \calR'_{\lfloor(c' + c_1)n\rfloor} \big) \ge c\sqrt{n} \big)
	\le \big(c \sqrt{n}\big)^{-1}\bbE\big[ G\big( \widehat\calR_{c_4 n}, \calR'_{c_4n} \big)\big] 
	\le c_5 n^{-1/2} h_d(n).
	\end{align*} 	 
	 Since for $d/\alpha>5 /2$ the index of $h_d(n)$ is less than $1/2$,
	 the last term tends to zero and the result follows.
\end{proof}

In what follows, we use the notation $\widetilde\calC[m, n] = \Capa(\widetilde\calR[m, n])$, $\widehat{\calC}[m, n] = \Capa(\widehat{\calR}[m, n])$, $\widetilde\calC_n = \Capa(\widetilde\calR_n)$ and  $\widehat{\calC}_n = \Capa(\widehat{\calR}_n)$, $m,n\ge0$.

\begin{lemma}\label{lm:lln_for_C_tilde}
	Assume (\textbf{A4}) and that $\{S_n\}_{n\ge0}$ (or $\{\widehat S_n\}_{n\ge0}$) is transient.
	Then for any $k\geq 0$
	\begin{equation*}
	\lim_{n\nearrow \infty} \frac{\widetilde{C}[k, k + n]}{n}
	=\frac{\mu_d}{1 - p}\qquad \bbP\text{-a.s.},
	\end{equation*}
	where $\mu_d$ is the limit from \eqref{Jain_Orey}.
\end{lemma}

\begin{proof}
	Since $\widetilde{S}_k = \widehat S_{k + N_{k - 1}}$ and $\widetilde{S}_{k + n} = \widehat S_{k + n + N_{k + n}}$,  we have
	\begin{equation*}
	\widetilde{\calR}[k, k + n] = \widehat\calR[k + N_{k - 1}, k + n + N_{k + n}],\qquad n,k\ge0.
	\end{equation*}
	Observe that $N_{k + n} = N_{k - 1} + N[k, k + n]$, where $N[k, k+n] = \sum_{i = k}^{k + n} \xi_i$. 
	Hence
	\begin{align*}
	\widetilde{\calC}[k,\,  k + n] = \widehat\calC\big[ k + N_{k - 1}, \, k + n + N_{k - 1} + N[k, k + n]\big].
	\end{align*}
	By the strong law of large numbers
	\begin{equation*}
	\lim_{n\nearrow \infty} \frac{N[k, k+n]}{n} = \frac{p}{1 - p}\qquad \bbP\text{-a.s.}
	\end{equation*}
	Therefore
	\begin{equation*}
	\lim_{n\nearrow \infty} \frac{\widetilde{\calC}[k, k + n]}{n} 
	=\frac{\mu_d}{1 - p}\qquad \bbP\text{-a.s.}
	\end{equation*}
	as desired.
\end{proof}

We finally prove that $\sigma_d$ is strictly positive. 

\begin{lemma}\label{positivity} Assume (\textbf{A1})-(\textbf{A4}) and
	 $d/\alpha > 5 /2$. Then $\sigma_d^2>0$.
\end{lemma}

\begin{proof}
	We define three sequences
	\begin{align*}
	i_n = \floor{(1 - p)(n - A\sqrt{n})},\quad j_n = \floor{(1 - p)n},\quad 
	k_n = \floor{(1 - p)(n + A\sqrt{n})},\qquad n\ge1,
	\end{align*}
	for a constant $A>0$ which will be specified later.  Lemma \ref{lm:lln_for_C_tilde} implies
	\begin{equation*}
	\lim_{n\nearrow \infty} \frac{\widetilde{\calC}[j_n, k_n]}{\sqrt{n}} = A\, \mu_d\qquad \bbP\text{-a.s.}
	\end{equation*}
	Thus, for $n$ large enough
	\begin{equation}\label{eq:cap-jn-kn}
	\bbP \big(  \widetilde{C}[j_n, k_n] \ge 3A \mu_d \sqrt{n} /4\big) \ge \frac{3}{4}.
	\end{equation}
	Similarly, we show that for $n$ large enough
	\begin{equation}\label{eq:cap-in-jn}
	\bbP \big(\widetilde{C}[i_n, j_n] \ge 3 A \mu_d \sqrt{n} /4 \big) \ge \frac{3}{4}.
	\end{equation}
	By Lemma \ref{lm:sum_sum_G_ge_sqrt_n} we get for $n$ large enough
	\begin{equation}\label{eq:G-jn-kn}
	\bbP\Big(G\big( \widetilde{\calR}[0, j_n] ,  \widetilde{\calR}[j_n, k_n]\big)   >  A \mu_d \sqrt{n}/8\Big) \le \frac{1}{8}
	\end{equation}
	and
	\begin{equation}\label{eq:G-in-jn}
	\bbP\Big(G \big( \widetilde{\calR}[0, i_n],  \widetilde{\calR}[i_n, j_n]\big)  > A \mu_d \sqrt{n}/8 \Big) \le \frac{1}{8}.
	\end{equation}
	We introduce the following events
	\begin{align*}
	B_n & = \vit{\frac{N_{i_n - 1} - \bbE[N_{i_n - 1}]}{\sqrt{n}} \in [A + 1, A + 2]},\\
	D_n & = \vit{\frac{N_{k_n - 1} - \bbE[N_{k_n - 1}]}{\sqrt{n}} \in [1 - A, 2 - A]}.
	\end{align*}
	By the central limit theorem, there exists a constant $c_A > 0$ such that 
	$
	\bbP(B_n) \ge c_A$ and $\bbP(D_n) \ge c_A$	
	for $n$ large enough. We distinguish between two cases: 
	\begin{itemize}
		\item [(i)] $\bbP(\widetilde{\calC}_{j_n} \ge \bbE[\widehat\calC_n]) \ge 1/2$;

		\item[(ii)]$ \bbP(\widetilde{\calC}_{j_n} \le \bbE[\widehat\calC_n]) \ge 1/2.$
	\end{itemize}
	We first study case (i). By Lemma \ref{capacity_decomp} we have
	\begin{equation*}
	\widetilde{\calC}[0, k_n] \ge \widetilde{\calC}[0, j_n] + \widetilde{\calC}[j_n, k_n] - 2 G\big( \widetilde{\calR}[0, j_n] , \widetilde{\calR}[j_n, k_n]\big) .
	\end{equation*}
	We thus obtain
	\begin{align*}
	\bbP \Big(\widetilde{\calC}[0, k_n] \ge \bbE[\widehat\calC_n] + A \mu_d \sqrt{n}/2 \Big)
	& \ge \bbP\Big( \widetilde{\calC}[0, j_n] \ge \bbE[\widehat\calC_n],\, \widetilde{\calC}[j_n, k_n] \ge 3 A \mu_d \sqrt{n}/4 \Big) \nonumber \\
	& \quad - \bbP\Big(G\big( \widetilde{\calR}[0, j_n] , \widetilde{\calR}[j_n, k_n]\big) >  A \mu_d \sqrt{n}/8 \Big).
	\end{align*}
	In view of the assumption, space homogeneity of the capacity and \eqref{eq:cap-jn-kn} we have that
	\begin{align*}
	&\bbP\Big(\widetilde{\calC}[0, j_n] \ge \bbE[\widehat\calC_n], \widetilde{\calC}[j_n, k_n] \ge 3 A \mu_d \sqrt{n}/4 \Big)\\
&=	\bbP\Big(\mathrm{Cap}(\widetilde\calR_{j_n}-\widetilde S_{j_n}) \ge \bbE[\widehat\calC_n], \mathrm{Cap}(\widetilde{\calR}[j_n, k_n]-\widetilde S_{j_n}) \ge 3 A \mu_d \sqrt{n}/4 \Big)\\
&=\bbP\Big(\mathrm{Cap}(\widetilde\calR_{j_n}-\widetilde S_{j_n}) \ge \bbE[\widehat\calC_n]\Big) \bbP\Big(\mathrm{Cap}(\widetilde{\calR}[j_n, k_n]-\widetilde S_{j_n}) \ge 3 A \mu_d \sqrt{n}/4 \Big)\\
&=\bbP\Big(\widetilde{\calC}[0, j_n] \ge \bbE[\widehat\calC_n]\Big)\bbP\Big( \widetilde{\calC}[j_n, k_n] \ge 3 A \mu_d \sqrt{n}/4 \Big)	\ge \frac{3}{8}.
	\end{align*} 
	This  together with \eqref{eq:G-jn-kn} implies
	\begin{equation*}
	\bbP\Big(\widetilde{\calC}[0, k_n] \ge \bbE[\widehat\calC_n] + \frac{1}{2} A \mu_d \sqrt{n}\Big) \ge \frac{1}{4}.
	\end{equation*} 
	By independence of $\{N_n\}_{n\ge-1}$ and  $\{\widetilde{S}_n\}_{n\ge0}$ we get
	\begin{equation*}
	\bbP \Big(\widetilde{\calC}[0, k_n] \ge \bbE[\widehat\calC_n] + \frac{1}{2} A \mu_d \sqrt{n}, D_n\Big) \ge \frac{c_A}{4}.
	\end{equation*} 
	We next observe that on $D_n$ we have $k_n + N_{k_n - 1} \in [n, n + 2\sqrt{n}]$.  
	We also recall that 
	\begin{align*}
	\widetilde{\calR}[0, k_n] = \widehat\calR[0, k_n + N_{k_n - 1}],
	\end{align*}	
	and whence
	\begin{equation*}
	\bbP \Big(\exists\, m \le 2\sqrt{n} : \widehat\calC[0, n + m] \ge \bbE[\widehat\calC_n] + \frac{1}{2} A \mu_d \sqrt{n}\Big) \ge \frac{c_A}{4}.
	\end{equation*} 
	Since $\{\widehat\calC_n\}_{n\ge0}$ is clearly increasing in $n$, we deduce that
	\begin{equation*}
	\bbP \Big(\widehat\calC[0, n + 2\sqrt{n}] \ge \bbE[\widehat\calC_n] + \frac{1}{2} A \mu_d \sqrt{n}\Big) \ge \frac{c_A}{4}.
	\end{equation*} 
	and, finally, the deterministic bound $\widehat\calC_{n + 2\sqrt{n}} \le \widehat\calC_n + 2\sqrt{n}$ yields
	\begin{equation*}
	\bbP \Big(\widehat\calC_n \ge \bbE[\widehat\calC_n] + ( A \mu_d/2 - 2) \sqrt{n}\Big) \ge \frac{c_A}{4}.
	\end{equation*} 
	Choosing $A$ large enough such that $A\mu_d/2 - 2 > 0$ and applying the  Chebyshev's inequality shows that in case (i) we have
	\begin{equation*}
	\Var(\calC_n) =\Var(\widehat\calC_n)\ge cn,
	\end{equation*}
	as desired.

	In case (ii) we proceed similarly. By Lemma \ref{capacity_decomp}, we have
	\begin{equation*}
	\widetilde{\calC}[0, i_n] \le \widetilde{\calC}[0, j_n] - \widetilde{\calC}[i_n, j_n] + 2 G\big( \widetilde{\calR}[0, i_n], \widetilde{\calR}[i_n, j_n]\big) .
	\end{equation*}
	Next, equations \eqref{eq:cap-in-jn}, \eqref{eq:G-in-jn} and the fact that $\bbP(B_n) \ge c_A $ imply
	\begin{equation*}
	\bbP\Big( \widetilde{\calC}[0, i_n] \le \bbE[\widehat\calC_n] - \frac{1}{2} A \mu_d \sqrt{n}, B_n\Big) \ge \frac{c_A}{4}.
	\end{equation*} 
	On $B_n$ we have $i_n + N_{i_n - 1} \in [n, n + 2\sqrt{n}]$ and it follows that
	\begin{equation*}
	\bbP \Big(\exists\, m \le 2\sqrt{n} : \widehat\calC[0, n + m] \le \bbE[\widehat\calC_n] - \frac{1}{2} A \mu_d \sqrt{n}\Big) \ge \frac{c_A}{4}.
	\end{equation*}
	We thus finally conclude that
	\begin{equation*}
	\bbP \Big(\widehat\calC_n \le \bbE[\widehat\calC_n] - \frac{1}{2} A \mu_d \sqrt{n}\Big) \ge \frac{c_A}{4}
	\end{equation*} 
	and an application of the Chebyshev's inequality finishes the proof.
\end{proof}

\subsection*{Proof of Theorem \ref{CLT}}

We start with the following technical lemma.

\begin{lemma}\label{lm:norm_4_of_Cn}
	Assume (\textbf{A1})-(\textbf{A3}) and $d/\alpha > 5 /2$.   Then there is a constant $C > 0$ such that 
	\begin{equation*}
	\bbE\big[  \overline{\calC}_n^4 \big] \le Cn^2,\qquad n\ge1,
	\end{equation*} where $\overline\calC_n=\calC_n-\bbE[\calC_n].$
\end{lemma}
\begin{proof}
	The proof is similar to that of Lemma \ref{lm:existence_of_Varn/n_limit}. For $k \ge 1$ we set
	\begin{equation*}
	\alpha_k = \sup \left\{  \Vert \overline{\calC}_n\Vert_4:\, 2^k \le n \le 2^{k + 1}\right\},
	\end{equation*}
	where $\lVert\cdot\rVert_4=\bbE[(\cdot)^4]^{1/4}.$
	For $k \ge 2$ we take $2^k \le n < 2^{k + 1}$ and we set $l = \floor{n/2}$ and $m = n - l$. Using Lemma \ref{capacity_decomp}, as in the proof of Lemma \ref{lm:existence_of_Varn/n_limit}, we obtain
	\begin{equation}\label{eq:norm4_Cn_bdd}
	\Vert\overline{\calC}_n\Vert_4 \le \Vert\overline{\calC}^{(1)}_{l} + \overline{\calC}^{(2)}_{m}\Vert_4 + 4\Vert G(\calR^{(1)}_l,\calR^{(2)}_m) \Vert_4,
	\end{equation}
	where again $\calC_l^{(1)}$ and $\calC_m^{(2)}$ ($\calR_l^{(1)}$ and $\calR_m^{(2)}$) are independent and have the same law as $\calC_l$ and $\calC_m$ ($\calR_l$ and $\calR_m$), respectively.
	We observe that
	\begin{align*}
	\bbE\left[\left(\overline{\calC}^{(1)}_{l} + \overline{\calC}^{(2)}_{m}\right)^4\right]
		&=  \bbE\left[\left(\overline{\calC}^{(1)}_{l}\right)^4\right] + \bbE\left[\left(\overline{\calC}^{(2)}_{m}\right)^4\right] + 6 \bbE\left[\left(\overline{\calC}^{(1)}_{l}\right)^2\right] \bbE\left[\left(\overline{\calC}^{(2)}_{m}\right)^2\right],
	\end{align*}
	where we used the fact that $\overline{\calC}^{(1)}_{l}$ and $\overline{\calC}^{(2)}_{m}$ are two independent and centered random variables. From Lemma \ref{lm:existence_of_Varn/n_limit}   we have
	\begin{equation*}
	\bbE\left[\left(\overline{\calC}^{(1)}_{l}\right)^2\right] \bbE\left[\left(\overline{\calC}^{(2)}_{m}\right)^2\right]\le c_1n^2,
	\end{equation*}
	whereas, by Lemma \ref{lm:E_of_sum_sum_G^k} and the fact that $h_d(n)$ is regularly varying of index less that $1/2$,
	\begin{align*}
	\Vert G(\calR^{(1)}_l,\calR^{(2)}_m) \Vert_4 \le \Vert G(\calR^{(1)}_n,\calR^{(2)}_n) \Vert_4 \le c_2 h_d(n)\leq c_3\sqrt{n}.
	\end{align*}
	Combining this with the elementary inequality $(a + b)^{1/4} \le a^{1/4} + b^{1/4}$, $a,b\ge0$, 
	we get
	\begin{equation*}
	\Vert\overline{\calC}_n\Vert_4 \le \left(\bbE\left[\left(\overline{\calC}^{(1)}_{l}\right)^4\right] + \bbE\left[\left(\overline{\calC}^{(2)}_{m}\right)^4\right]\right)^{1/4} + c_4 \sqrt{n}.
	\end{equation*}
	Similarly as in Lemma \ref{lm:existence_of_Varn/n_limit} we thus obtained
	\begin{equation*}
	\alpha_k \le 2^{1/4} \alpha_{k - 1} + c_4  2^{k/2}.
	\end{equation*}
	Setting $\beta_k = 2^{-k/2} \alpha_k$, $k\ge1$,  we deduce that
	\begin{equation*}
	\beta_k \le \frac{1}{2^{1/4}} \beta_{k - 1} + c_4,
	\end{equation*}
	 which shows that $\{\beta_k\}_{k\ge1}$ is a bounded sequence. Therefore, $\alpha_k \le c_5 2^{k/2}$, $k\ge1$,  which immediately yields the result.
\end{proof}

The proof of Theorem \ref{CLT} is based on the  dyadic capacity decomposition formula derived in \cite[Corollary 2.1]{Asselah_Zd} and Lindeberg-Feller central limit theorem, which we state  for reader's convenience.

\begin{lemma}[{\cite[Corollary 2.1]{Asselah_Zd}}]\label{dyadic} Let $L,n\ge1$ be such that $2^L\le n$.  Then, \begin{equation}\label{eq:dyadic}\sum_{i=1}^{2^L}\mathrm{Cap}\,(\calR^{(i)}_{n/2^L})-2\sum_{l=1}^L\sum_{i=1}^{2^{l-1}}\calE_l^{(i)}
		\le \calC_n\le \sum_{i=1}^{2^L}\mathrm{Cap}\,(\calR^{(i)}_{n/2^L}),\end{equation} where $\{\calR^{(i)}_{n/2^L}\}_{i=1,\dots,2^L}$ are independent and  $\calR^{(i)}_{n/2^L}$ has the same law as $\calR_{\floor{n/2^L}}$ or $\calR_{\floor{n/2^L+1}}$, and for each $l=1,\dots,2^L$ the random variables $\{\calE_l^{(i)}\}_{i=1,\dots,2^{l-1}}$
		are independent and $\calE_l^{(i)}$ has the same law as $G(\calR^{(i)}_{n/2^l},\bar\calR^{(i)}_{n/2^l})$ with  $\{\bar\calR_n\}_{n\ge0}$ being an independent copy of $\{\calR_n\}_{n\ge0}$.
\end{lemma}

\begin{lemma}[{\cite[Theorem 4.5]{Durrett}}]\label{feller}
	For each $n\ge1$ let $\{X_{n,m}\}_{1\le m\le n}$ be a sequence of independent random variables. If
	\begin{itemize}
		\item [(i)] $\sum_{m=1}^n\Var(X_{n,m})\xrightarrow[]{n\nearrow\infty}\sigma^2>0;$
		
		\medskip
		
		\item [(ii)] for every $\varepsilon>0$, $\sum_{m=1}^n\bbE\left[(X_{n,m}-\bbE[X_{n,m}])^2\bbjedan_{\{|X_{n,m}-\bbE[X_{n,m}|>\varepsilon\}}\right]\xrightarrow[]{n\nearrow\infty}0,$
	\end{itemize}
	then $X_{n,1}+\cdots+ X_{n,n}\xrightarrow[n\nearrow\infty]{\rm{(d)}}\sigma\,\calN(0,1).$
	\end{lemma}

Finally, we prove  Theorem \ref{CLT}.

\begin{proof}[Proof of  Theorem \ref{CLT}]
Denote $\calC^{(i)}_{n/2^L}= \mathrm{Cap}\,(\calR^{(i)}_{n/2^L})$, $i=1,\dots,2^L$.
	By taking expectation in \eqref{eq:dyadic} and then subtracting those two relations we obtain	
	\begin{equation}\label{eq:relation_from_Cor_2.1-with_bars}
	\sum_{i = 1}^{2^L} \overline\calC^{(i)}_{n/2^L} - 2 \sum_{l = 1}^L \sum_{i = 1}^{2^{l - 1}} \calE_l^{(i)} \le \overline{\calC}_n \le \sum_{i = 1}^{2^L} \overline\calC^{(i)}_{n/2^L} + 2 \sum_{l = 1}^L \sum_{i = 1}^{2^{l - 1}} \bbE[\calE_l^{(i)}].
	\end{equation}
	Further, define
	\begin{equation*}
	\calE(n) = \sum_{i = 1}^{2^L} \overline\calC^{(i)}_{n/2^L} - \overline{\calC}_n, \qquad n\ge1.
	\end{equation*}
	Using \eqref{eq:relation_from_Cor_2.1-with_bars} and Lemma \ref{lm:E_of_sum_sum_G^k}, we get that
	\begin{equation*}
	\bbE\big[ \aps{\calE(n)} \big] \le 4 \bbE\ugl{\sum_{l = 1}^L \sum_{i = 1}^{2^{l - 1}} \calE_l^{(i)}} \le c_1\sum_{l = 1}^L \sum_{i = 1}^{2^{l - 1}} h_d \obl{\frac{n}{2^l}} \le c_2  h_d(n) \sum_{l = 1}^L 2^{l-1} \le c_22^{L} h_d(n).
	\end{equation*} 
Next we distinguish between two cases. If $5 /2<d/\alpha<3$ then we set $L = \lfloor\log_2\big(n^{\Delta /2} \big)\rfloor$, where $\Delta = d/\alpha - 5/2$. This implies 
	\begin{align}\label{conv123}
	\lim_{n\nearrow \infty} \frac{\bbE\big[\aps{\calE(n)}\big] }{\sqrt{n} } =0.
	\end{align}	
	If $d/\alpha\geq 3$ then $h_d(n) $ is slowly varying and in this case it is enough to choose $L =\lfloor\log_2\big(n^{1/4} \big)\rfloor $ to obtain \eqref{conv123}.
	We are thus left to prove that
	\begin{equation*}
	\sum_{i = 1}^{2^L} \frac{\overline\calC^{(i)}_{n/2^L}}{\sqrt{n}}\xrightarrow[n\nearrow\infty]{\text{(d)}}  \sigma_d\, \calN(0, 1),
	\end{equation*} 
	where $\sigma_d>0$ is from Lemma \ref{lm:existence_of_Varn/n_limit}.
	To establish this result we apply the Lindeberg-Feller central limit theorem.
	By Lemma \ref{lm:existence_of_Varn/n_limit} we have
	\begin{equation*}
	\lim_{n\nearrow \infty }\sum_{i = 1}^{2^L} \frac{1}{n} \Var(\overline\calC^{(i)}_{n/2^L}) = \sigma_d^2,
	\end{equation*}
	which means that the first Lindeberg-Feller condition is satisfied. 
	It remains to check that for any $\varepsilon >0$ it holds that
	\begin{equation*}
	\lim_{n \nearrow \infty} \sum_{i = 1}^{2^L} \frac{1}{n} \bbE\left[ \left(\overline\calC^{(i)}_{n/2^L}\right)^2 \bbjedan_{\{ \aps{\overline\calC^{(i)}_{n/2^L}} > \varepsilon \sqrt{n} \}}\right] = 0.
	\end{equation*}
	Observe that by the Cauchy-Schwartz inequality we have 
	\begin{align*}
	\bbE\left[ \left(\overline\calC^{(i)}_{n/2^L}\right)^2 \bbjedan_{\{ \aps{\overline\calC^{(i)}_{n/2^L}} > \varepsilon \sqrt{n} \}}\right]
	&\le \left(\bbE\left[\left(\overline\calC^{(i)}_{n/2^L}\right)^4\right] \bbP(\aps{\overline\calC^{(i)}_{n/2^L}} > \varepsilon \sqrt{n})\right)^{1/2}.
	\end{align*}
Further, the Chebyshev inequality combined with Lemma \ref{lm:existence_of_Varn/n_limit}, strict positivity of $\sigma_d$ and Lemma \ref{lm:norm_4_of_Cn}, imply
	\begin{align*}
\bbE\left[\left(\overline\calC^{(i)}_{n/2^L}\right)^4\right] \bbP(\aps{\overline\calC^{(i)}_{n/2^L}} > \varepsilon \sqrt{n})
	 \le c_3\obl{\frac{n}{2^L}}^2 \, \frac{\Var\left(\overline\calC^{(i)}_{n/2^L}\right)}{\varepsilon^2 n} 
	 \le c_4 \frac{n^2}{\varepsilon^2 2^{3L}} .
	\end{align*}
	Based on the choice $L = \lfloor\log_2\big(n^{\Delta /2} \big)\rfloor$ if $5 /2<d/\alpha<3$ (and $L =\lfloor\log_2\big(n^{1/4} \big)\rfloor $ if $d/\alpha\geq 3$),  we conclude that in both cases 
	\begin{equation*}
	\sum_{i = 1}^{2^L} \frac{1}{n} \bbE\left[ \left(\overline\calC^{(i)}_{n/2^L}\right)^2 \bbjedan_{\{ \aps{\overline\calC^{(i)}_{n/2^L}} > \varepsilon \sqrt{n} \}}\right] \le \frac{c_5}{\varepsilon 2^{L/2}} \xrightarrow[]{n\nearrow\infty} 0,
	\end{equation*}
and this finishes the proof.
\end{proof}

\section*{Acknowledgement} 
\noindent
We wish to thank R.\ L.\ Schilling for stimulating discussions.
This work has been supported by \textit{Deutscher Akademischer} \textit{Austauschdienst} (DAAD) and \textit{Ministry of Science and Education of the Republic of Croatia} (MSE) via project \textit{Random Time-Change and Jump Processes}.  Financial support through the \textit{Croatian Science Foundation} under projects 8958 and  4197 (for N. Sandri\'c),  and the \textit{Austrian Science Fund} (FWF) under project P31889-N35 and \textit{Croatian Science Foundation} under project 4197 (for S.\ \v Sebek) is gratefully acknowledged.
We also thank the anonymous referees for helpful comments that have led to improvements of the presentation of  the article.

\bibliographystyle{babamspl}
\bibliography{Capacity_of_the_range_of_stable_random_walks}

\end{document}